\documentclass{amsart}
\usepackage{amsmath,amssymb}
\usepackage[dvips]{graphics}
\usepackage[all]{xy}
\newtheorem{Theorem}{Theorem}[section]
\newtheorem{Lemma}[Theorem]{Lemma}
\newtheorem{Proposition}[Theorem]{Proposition}
\newtheorem{Corollary}[Theorem]{Corollary}

\theoremstyle{definition}

\newtheorem{Example}[Theorem]{Example}

\theoremstyle{remark}

\def\labelenumi{\rm ({\makebox[0.8em][c]{\theenumi}})}
\def\theenumi{\arabic{enumi}}
\def\labelenumii{\rm {\makebox[0.8em][c]{(\theenumii)}}}
\def\theenumii{\roman{enumii}}

\renewcommand{\thefigure}
{\arabic{section}.\arabic{figure}}
\makeatletter
\@addtoreset{figure}{section}
\def\@thmcountersep{-}
\makeatother

\numberwithin{equation}{section}

\begin{document} 

\title{A refinement of the Conway-Gordon theorems}

\author{Ryo Nikkuni}
\address{Department of Mathematics, School of Arts and Sciences, Tokyo Woman's Christian University, 2-6-1 Zempukuji, Suginami-ku, Tokyo 167-8585, Japan}
\email{nick@lab.twcu.ac.jp}
\thanks{The author was partially supported by Grant-in-Aid for Young Scientists (B) (No. 21740046), Japan Society for the Promotion of Science.}


\date{}

\dedicatory{Dedicated to Professor Tohl Asoh for his 60th birthday}


\begin{abstract}
In 1983, Conway-Gordon showed that for every spatial complete graph on $6$ vertices, the sum of the linking numbers over all of the constituent $2$-component links is congruent to $1$ modulo $2$, and for every spatial complete graph on $7$ vertices, the sum of the Arf invariants over all of the Hamiltonian knots is also congruent to $1$ modulo $2$. In this article, we give integral lifts of the Conway-Gordon theorems above in terms of the square of the linking number and the second coefficient of the Conway polynomial. As applications, we give alternative topological proofs of theorems of Brown-Ram{\'\i}rez Alfons{\'\i}n and Huh-Jeon for rectilinear spatial complete graphs which were proved by computational and combinatorial methods. 
\end{abstract}

\maketitle

\section{Introduction} 

Throughout this paper we work in the piecewise linear category. Let $f$ be an embedding of a finite graph $G$ into ${\mathbb R}^{3}$. Then $f$ (or $f(G)$) is called a {\it spatial embedding} of $G$ or simply a {\it spatial graph}. Two spatial embeddings $f$ and $g$ of $G$ are said to be {\it ambient isotopic} if there exists an orientation-preserving self homeomorphism $\Phi$ on ${\mathbb R}^{3}$ such that $\Phi\circ f=g$. We call a subgraph $\gamma$ of $G$ which is homeomorphic to the circle a {\it cycle} of $G$, and a cycle of $G$ which contains exactly $k$ edges a {\it $k$-cycle} of $G$. We denote the set of all cycles of $G$, the set of all $k$-cycles of $G$ and the set of all pairs of two disjoint cycles consisting of a $k$-cycle and an $l$-cycle of $G$ by $\Gamma(G)$, $\Gamma_{k}(G)$ and $\Gamma_{k,l}(G)$, respectively. For $\gamma\in \Gamma(G)$ (resp. $\lambda\in \Gamma_{k,l}(G)$) and a spatial embedding $f$ of $G$, $f(\gamma)$ (resp. $f(\lambda)$) is none other than a knot (resp. $2$-component link) in $f(G)$. 

Let $K_{n}$ be the {\it complete graph} on $n$ vertices, namely the graph consisting of $n$ vertices $1,2,\ldots,n$, a pair of whose vertices $i$ and $j$ is connected by exactly one edge $\overline{ij}$ if $i\neq j$. Let us recall the following two famous theorems in spatial graph theory, which are called the Conway-Gordon theorems. 

\begin{Theorem}\label{CG1} {\rm (Conway-Gordon \cite{CG83})}
For any spatial embedding $f$ of $K_{6}$, it follows that  
\begin{eqnarray}\label{CG_f1}
\sum_{\lambda\in \Gamma_{3,3}(K_{6})}{\rm lk}(f(\lambda)) \equiv 1 \pmod{2}, 
\end{eqnarray}
where ${\rm lk}$ denotes the {\it linking number} in ${\mathbb R}^{3}$. In particular, for any spatial embedding $f$ of $K_{6}$, there exists a pair of two disjoint $3$-cycles $\lambda$ of $K_{6}$ such that $f(\lambda)$ is a non-splittable $2$-component link with odd linking number. 
\end{Theorem}

\begin{Theorem}\label{CG2} {\rm (Conway-Gordon \cite{CG83})}
For any spatial embedding $f$ of $K_{7}$, it follows that  
\begin{eqnarray}\label{CG_f2}
\sum_{\gamma\in \Gamma_{7}(K_{7})}{\rm Arf}(f(\gamma)) \equiv 1 \pmod{2}, 
\end{eqnarray}
where ${\rm Arf}$ denotes the {\it Arf invariant} \cite{Rober65}. In particular, for any spatial embedding $f$ of $K_{7}$, there exists a $7$-cycle $\gamma$ of $K_{7}$ such that $f(\gamma)$ is a non-trivial knot with Arf invariant one. 
\end{Theorem}

Conway-Gordon's formulas (\ref{CG_f1}) and (\ref{CG_f2}) hold only modulo two. Note that the square of the linking number is congruent to the linking number modulo two, and the second coefficient of the {\it Conway polynomial} of a knot is congruent to the Arf invariant modulo two \cite{Kauffman83}. Our first purpose in this article is to refine Conway-Gordon's formulas above by giving their integral lifts in terms of the square of the linking number and the second coefficient of the Conway polynomial as follows. 

\begin{Theorem}\label{main1} 
For any spatial embedding $f$ of $K_{6}$, we have that  
\begin{eqnarray*}
2\left\{
\sum_{\gamma\in \Gamma_{6}(K_{6})}a_{2}(f(\gamma))
-\sum_{\gamma\in \Gamma_{5}(K_{6})}a_{2}(f(\gamma))
\right\}
=
\sum_{\lambda\in \Gamma_{3,3}(K_{6})}{{\rm lk}(f(\lambda))}^{2}-1,
\end{eqnarray*}
where $a_{2}$ denotes the second coefficient of the Conway polynomial. 
\end{Theorem}

\begin{Theorem}\label{main2} 
For any spatial embedding $f$ of $K_{7}$, we have that   
\begin{eqnarray*}
&&7\sum_{\gamma\in \Gamma_{7}(K_{7})}a_{2}(f(\gamma))
-6\sum_{\gamma\in \Gamma_{6}(K_{7})}a_{2}(f(\gamma))
-2\sum_{\gamma\in \Gamma_{5}(K_{7})}a_{2}(f(\gamma))\\
&=& 2\sum_{\lambda\in \Gamma_{4,3}(K_{7})}{{\rm lk}(f(\lambda))}^{2}
-21.
\end{eqnarray*}
\end{Theorem}

Note that Theorems \ref{CG1} and \ref{CG2} can be obtained from Theorems \ref{main1} and \ref{main2} respectively by taking the modulo two reduction. We also show that Theorem \ref{main2} can be divided into the following two formulas. 

\begin{Corollary}\label{main3} 
For any spatial embedding $f$ of $K_{7}$, we have that
\begin{eqnarray}
\label{cor2_1} &&
14\left\{
\sum_{\gamma\in \Gamma_{7}(K_{7})}a_{2}(f(\gamma))
-\sum_{\gamma\in \Gamma_{6}(K_{7})}a_{2}(f(\gamma))
\right\}
\\
\nonumber &=&
4\sum_{\lambda\in \Gamma_{4,3}(K_{7})}{{\rm lk}(f(\lambda))}^{2}
-\sum_{\lambda\in \Gamma_{3,3}(K_{7})}{{\rm lk}(f(\lambda))}^{2}
-35, \\
\label{cor2_2} &&7\left\{
\sum_{\gamma\in \Gamma_{7}(K_{7})}a_{2}(f(\gamma))
-2\sum_{\gamma\in \Gamma_{5}(K_{7})}a_{2}(f(\gamma))
\right\}\\
\nonumber &=&
2\sum_{\lambda\in \Gamma_{4,3}(K_{7})}{{\rm lk}(f(\lambda))}^{2}
+3\sum_{\lambda\in \Gamma_{3,3}(K_{7})}{{\rm lk}(f(\lambda))}^{2}
-42.
\end{eqnarray}
\end{Corollary}

Note that Theorem \ref{main2} can be recovered from (\ref{cor2_1}) and (\ref{cor2_2}). We give proofs of Theorems \ref{main1}, \ref{main2} and Corollary \ref{main3} in section $3$. We remark here that the second coefficient of the Conway polynomial and the square of the linking number are Vassiliev invariants \cite{Vass90} of order $2$ \cite{Natan95}, \cite{Murakami96}. See also \cite{OT01} for general results about Vassiliev invariants of knots and links in a spatial graph. 

Our second purpose in this article is to give applications of Theorems \ref{main1} and \ref{main2} to the theory of {\it rectilinear} spatial graphs. Here a spatial embedding $f$ of a graph $G$ is said to be rectilinear if for any edge $e$ of $G$, $f(e)$ is a straight line segment in ${\mathbb R}^{3}$. Note that for any positive integer $n$, there exists a rectilinear spatial embedding of $K_{n}$. Actually it can be constructed by taking $n$ vertices on the moment curve $(t,t^{2},t^{3})$ in ${\mathbb R}^{3}$ and connecting every pair of two distinct vertices by a straight line segment (this also implies that any simple graph has a rectilinear spatial embedding). Note that every knot or link contained in a rectilinear embedding of $K_{n}$ has {\it stick number} less than or equal to $n$, where the stick number $s(L)$ of a link (or a knot) $L$ is the minimum number of edges in a polygon which represents $L$. The following are fundamental results on stick numbers for knots and $2$-component links, see \cite{Negami91}, \cite{ABGW97}, \cite{Mater02}. 

\begin{Proposition}\label{stick} 
{\rm (1)} For any non-trivial knot $K$, it follows that $s(K)\ge 6$. Moreover, $s(K)=6$ if and only if $K$ is a trefoil knot, and $s(K)=7$ if and only if $K$ is a figure eight knot. 

\noindent
{\rm (2)} For a $2$-component link $L$, it follows that $s(L)\ge 6$. Moreover, $s(L)=6$ if and only if $L$ is either a trivial link or a Hopf link, and $s(L)=7$ if and only if $L$ is a $(2,4)$-torus link. 
\end{Proposition}

Therefore we have that non-trivial knots in a rectilinear spatial embedding of $K_{6}$ (resp. $K_{7}$) are only trefoil knots (resp. trefoil knots and figure eight knots), and non-trivial $2$-component links in a rectilinear spatial embedding of $K_{6}$ (resp. $K_{7}$) are only Hopf links (resp. Hopf links and $(2,4)$-torus links). Then, as a result for rectilinear spatial embeddings of $K_{7}$ corresponding to Theorem \ref{CG2}, the following is known. 

\begin{Theorem}\label{Alfon} {\rm (Brown \cite{B77}, Ram{\'\i}rez Alfons{\'\i}n \cite{RA99})} 
For any rectilinear spatial embedding $f$ of $K_{7}$, $f(K_{7})$ contains a trefoil knot.   
\end{Theorem}

We remark here that Brown's paper \cite{B77} is unpublished, and Ram{\'\i}rez Alfons{\'\i}n's proof in \cite{RA99} is done by applying oriented matroid theory with the help of a computer. We refine Theorem \ref{Alfon} by an application of Theorems \ref{main1} and \ref{main2} as follows. 

\begin{Theorem}\label{main4} 
For any rectilinear spatial embedding $f$ of $K_{7}$, $\sum_{\gamma\in \Gamma_{7}(K_{7})}a_{2}(f(\gamma))$ is a positive odd integer. In particular, $\sum_{\gamma\in \Gamma_{7}(K_{7})}a_{2}(f(\gamma))=1$ if and only if the non-trivial $2$-component links in $f(K_{7})$ are exactly twenty one Hopf links. 
\end{Theorem}

By Theorem \ref{main4}, for any rectilinear spatial embedding $f$ of $K_{7}$, there exists a $7$-cycle $\gamma$ of $K_{7}$ such that $a_{2}(f(\gamma))>0$. Since $s(f(\gamma))\le 7$, by Proposition \ref{stick} (1), $f(\gamma)$ must be a trefoil knot because $a_{2}({\rm trefoil\ knot})=1$ and $a_{2}({\rm figure\ eight\ knot})=-1$. Namely Theorem \ref{Alfon} is obtained from Theorem \ref{main4} as a corollary. We give a proof of Theorem \ref{main4} in section $4$. Note that our proof is topological whereas Ram{\'\i}rez Alfons{\'\i}n's proof is extremely combinatorial and computational. 

On the other hand, as a result for rectilinear spatial embeddings of $K_{6}$ corresponding to Theorem \ref{CG1}, the following is known. 

\begin{Theorem}\label{HJ} {\rm (Huh-Jeon \cite{HJ07})} 
For any rectilinear spatial embedding $f$ of $K_{6}$, $f(K_{6})$ contains at most one trefoil knot and at most three Hopf links. In particular, 

\noindent
{\rm (1)} $f(K_{6})$ does not contain a trefoil knot if and only if $f(K_{6})$ contains exactly one Hopf link. 

\noindent
{\rm (2)} $f(K_{6})$ contains a trefoil knot if and only if $f(K_{6})$ contains exactly three Hopf links. 
\end{Theorem}

Although Huh and Jeon do not use a computer in \cite{HJ07}, their proof is in the same spirit as Ram{\'\i}rez Alfons{\'\i}n's in that it is also purely combinatorial. As an application of Theorem \ref{main1}, we give an alternative topological proof of Theorem \ref{HJ} in section 4.

\section{Spatial graph-homology invariants} 

In this section, we introduce some homological invariants of spatial graphs which are needed later. Let $G$ be $K_{5}$ or $K_{3,3}$, where $K_{3,3}$ denotes the {\it complete bipartite graph} on $3+3$ vertices, namely the graph consisting of $6$ vertices $1,2,\ldots,6$ a pair of whose vertices $i$ and $j$ are connected by exactly one edge $\overline{ij}$ if $i+j$ is odd. We give an orientation to each edge of $G$ as illustrated in Fig. \ref{K5K33label}. For an unordered pair of disjoint edges $(x,y)$ of $K_{5}$, we define the sign $\varepsilon(x,y)$ by $\varepsilon(e_{i},e_{j})=1$, $\varepsilon(d_{k},d_{l})=-1$ and $\varepsilon(e_{i},d_{k})=-1$. For an unordered pair of disjoint edges $(x,y)$ of $K_{3,3}$, we also define the sign $\varepsilon(x,y)$ by $\varepsilon(c_{i},c_{j})=1$, $\varepsilon(b_{k},b_{l})=1$ and 
$\varepsilon(c_{i},b_{k})=1$ if $c_{i}$ and $b_{k}$ are parallel in Fig. \ref{K5K33label} and $-1$ if $c_{i}$ and $b_{k}$ are anti-parallel in Fig. \ref{K5K33label}. 
For a spatial embedding $f$ of $G$, we fix a regular diagram of $f$ and denote the sum of the signs of all crossing points between $f(x)$ and $f(y)$ by $l(f(x),f(y))$, where $(x,y)$ is an unordered pair of disjoint edges of $G$. Then, an integer ${\mathcal L}(f)$ defined by 
\begin{eqnarray*}
{\mathcal L}(f)=\sum_{(x,y)}\varepsilon(x,y)l(f(x),f(y)), 
\end{eqnarray*}
where the sum is taken over all unordered pairs of disjoint edges of $G$, is called the {\it Simon invariant} of $f$.

\begin{figure}[htbp]
      \begin{center}
\scalebox{0.425}{\includegraphics*{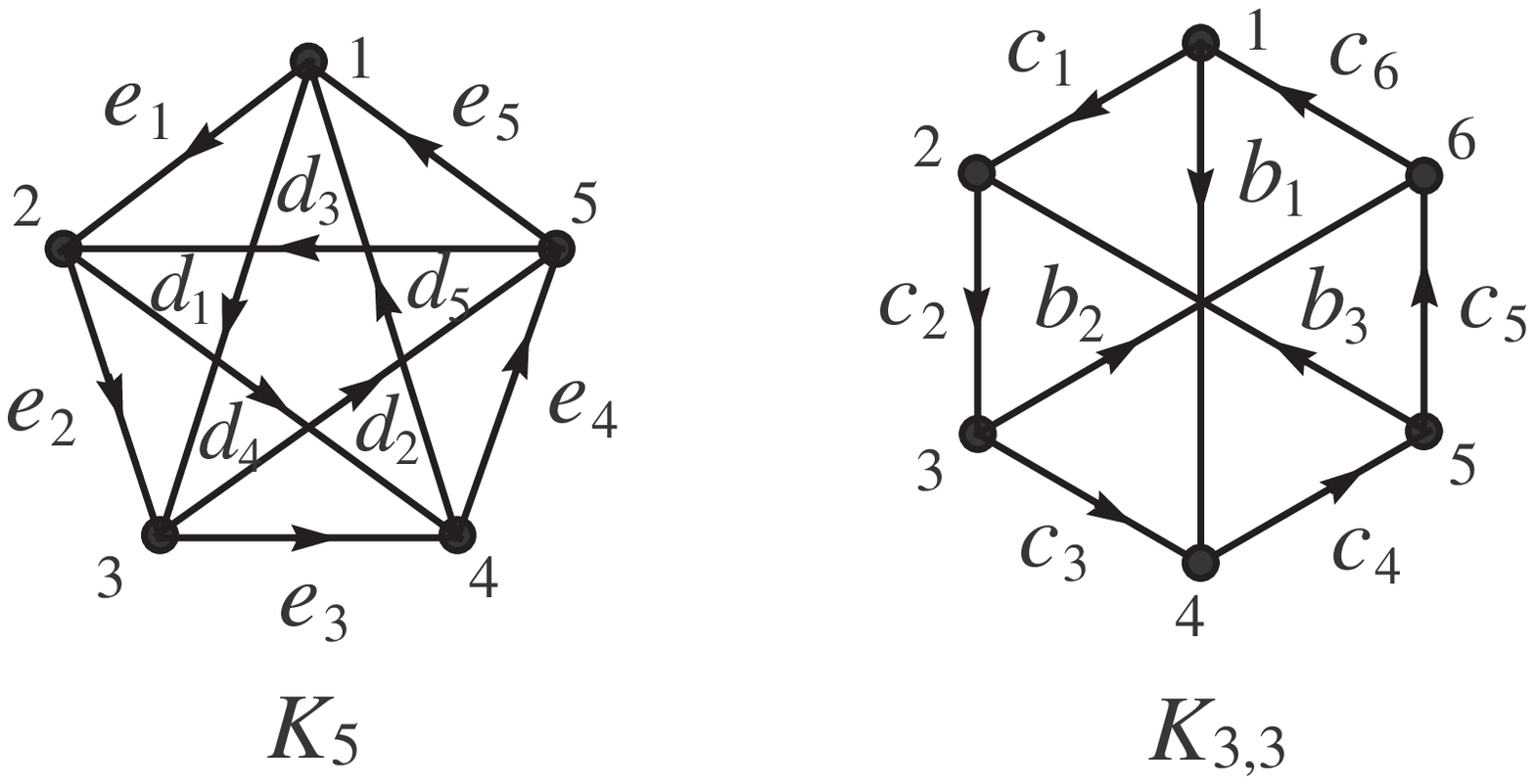}}
      \end{center}
   \caption{}
  \label{K5K33label}
\end{figure} 

It is known that ${\mathcal L}(f)$ is an odd integer valued ambient isotopy invariant of $f$ \cite{taniyama94b}. Moreover, the following is known. 

\begin{Proposition}\label{homo1} {\rm (\cite{taniyama94b})} 
Let $f$ be a spatial embedding of $K_{5}$ or $K_{3,3}$. Then ${\mathcal L}(f)$ is a spatial graph-homology invariant of $f$. 
\end{Proposition}

Here, a {\it spatial graph-homology} is an equivalence relation on spatial graphs introduced in \cite{taniyama94b} as a generalization of a {\it link-homology} on oriented links. We refer the reader to \cite{taniyama94b} for the precise definition of a spatial graph-homology. 

On the other hand, let $G$ be $K_{5}$, $K_{3,3}$ or $D_{4}$, where $D_{4}$ is the graph as illustrated in Fig. \ref{D4}. Let $\omega:\Gamma(G)\to {\mathbb Z}$ be a map defined by  
\begin{eqnarray*}
\omega(\gamma)=\left\{
       \begin{array}{@{\,}ll}
       1 & \mbox{(if $\gamma$ is a $5$-cycle)}\\
       -1 & \mbox{(if $\gamma$ is a $4$-cycle)}\\
       0 & \mbox{(if $\gamma$ is a $3$-cycle)}
       \end{array}
     \right.
\end{eqnarray*}
if $G=K_{5}$, 
\begin{eqnarray*}
\omega(\gamma)=\left\{
       \begin{array}{@{\,}ll}
       1 & \mbox{(if $\gamma$ is a $6$-cycle)}\\
       -1 & \mbox{(if $\gamma$ is a $4$-cycle)}\\
       \end{array}
     \right.
\end{eqnarray*}
if $G=K_{3,3}$ and 
\begin{eqnarray*}
\omega(\gamma)=\left\{
       \begin{array}{@{\,}ll}
       1 & \mbox{(if $\gamma$ is a $4$-cycle $e_{i}\cup e_{j}\cup e_{k}\cup e_{l}$ with $i+j+k+l\equiv 0\pmod{2}$)}\\
       -1 & \mbox{(if $\gamma$ is a $4$-cycle $e_{i}\cup e_{j}\cup e_{k}\cup e_{l}$ with $i+j+k+l\equiv 1\pmod{2}$)}\\
       0 & \mbox{(if $\gamma$ is a $2$-cycle)}
       \end{array}
     \right.
\end{eqnarray*}
if $G=D_{4}$. For a spatial embedding $f$ of $G$, an integer $\alpha_{\omega}(f)$ defined by 
\begin{eqnarray*}
\alpha_{\omega}(f)=\sum_{\gamma\in \Gamma(G)}\omega(\gamma)a_{2}(f(\gamma)) 
\end{eqnarray*}
is called the {\it $\alpha$-invariant} of $f$ \cite{taniyama94a}. Then the following holds. 

\begin{Proposition}\label{homo2} 
{\rm (1) (Motohashi-Taniyama \cite{MT97})} Let $f$ be a spatial embedding of $K_{5}$ or $K_{3,3}$. Then it follows that
\begin{eqnarray*}
\alpha_{\omega}(f)=\frac{{{\mathcal L}(f)}^{2}-1}{8}. 
\end{eqnarray*}

\noindent
{\rm (2) (Taniyama-Yasuhara \cite{TY01})} Let $f$ be a spatial embedding of $D_{4}$. We denote the pair of disjoint two $2$-cycles $e_{1}\cup e_{2}$ and $e_{5}\cup e_{6}$ $($resp. $e_{3}\cup e_{4}$ and $e_{7}\cup e_{8}$$)$ of $D_{4}$ by $\lambda$ $($resp. $\lambda'$$)$ Then it follows that 
\begin{eqnarray*}
\left|\alpha_{\omega}(f)\right|=\left|{\rm lk}(f(\lambda)){\rm lk}(f(\lambda'))\right|. 
\end{eqnarray*}
\end{Proposition}

\begin{figure}[htbp]
      \begin{center}
\scalebox{0.425}{\includegraphics*{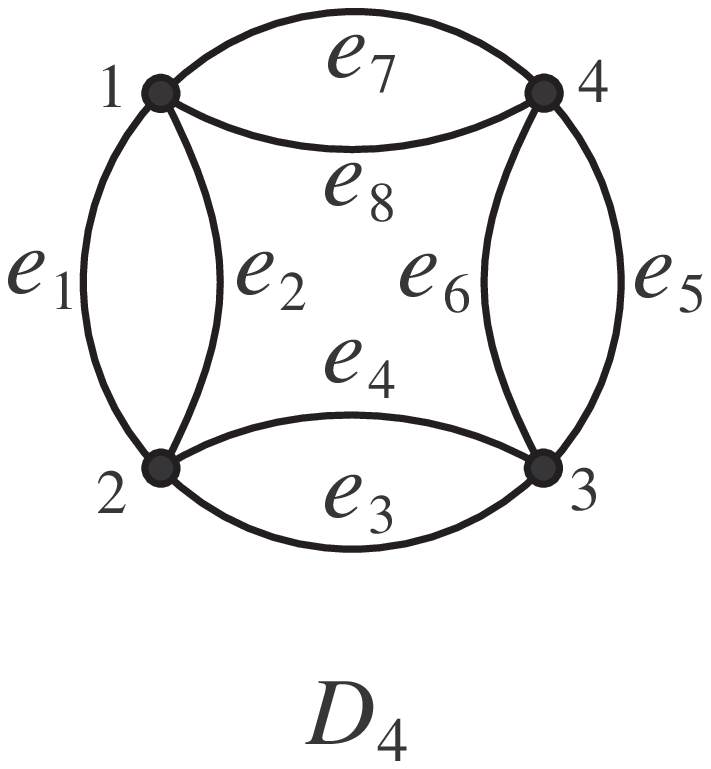}}
      \end{center}
   \caption{}
  \label{D4}
\end{figure} 
\section{Proofs of Theorems \ref{main1} and \ref{main2}} 

First we give a proof of Theorem \ref{main1}. We can see that there exist exactly ten subgraphs $H_{1},H_{2},\ldots,H_{10}$ of $K_{6}$ each of which is isomorphic to $K_{3,3}$, exactly six subgraphs $G_{1},G_{2},\ldots,G_{6}$ of $K_{6}$ each of which is isomorphic to $K_{5}$ and exactly ten pairs of disjoint $3$-cycles $\lambda_{1},\lambda_{2},\ldots,\lambda_{10}$ of $K_{6}$. Then we have the following. 

\begin{Lemma}\label{Simon_lemma} 
Let $f$ be a spatial embedding of $K_{6}$. Then we have that
\begin{eqnarray*}
\sum_{i=1}^{10}{\mathcal L}(f|_{H_{i}})^{2}
-\sum_{i=1}^{6}{\mathcal L}(f|_{G_{i}})^{2}
=4\sum_{i=1}^{10}{\rm lk}(f(\lambda_{i}))^{2}. 
\end{eqnarray*}
\end{Lemma}

\begin{proof}
For any spatial embedding $f$ of $K_{6}$, there exist ten integers $m_{i}\ (i=1,2,\ldots, 5)$ and $n_{i}\ (i=1,2,\ldots, 5)$ such that $f$ is spatial graph-homologous to the spatial embedding $h$ of $K_{6}$ as illustrated in Fig. \ref{K6homology} \cite{S96},\cite{N00}, where the rectangle represented by an integer $k$ stands for $|k|$ half twists as illustrated in Fig. \ref{twists}. In general, for two spatial embeddings $g_{1}$ and $g_{2}$ of a graph which are spatial graph-homologous, $g_{1}|_{F}$ and $g_{2}|_{F}$ are also spatial graph-homologous for any subgraph $F$ of the graph by the definition. We recall that the Simon invariant is a spatial graph-homology invariant (Proposition \ref{homo1}), and the linking number is also a typical spatial graph-homology invariant. Thus we have that 
\begin{eqnarray*}
{\mathcal L}(f|_{H_{i}})&=&{\mathcal L}(h|_{H_{i}})\ (i=1,2,\ldots,10),\\
{\mathcal L}(f|_{G_{i}})&=&{\mathcal L}(h|_{G_{i}})\ (i=1,2,\ldots,6),\\
{\rm lk}(f(\lambda_{i}))&=&{\rm lk}(h(\lambda_{i}))\ (i=1,2,\ldots,10). 
\end{eqnarray*}
Therefore it is sufficient to show that 
\begin{eqnarray}\label{f-h}
\sum_{i=1}^{10}{\mathcal L}(h|_{H_{i}})^{2}
-\sum_{i=1}^{6}{\mathcal L}(h|_{G_{i}})^{2}
=4\sum_{i=1}^{10}{\rm lk}(h(\lambda_{i}))^{2}. 
\end{eqnarray}

\begin{figure}[htbp]
      \begin{center}
\scalebox{0.475}{\includegraphics*{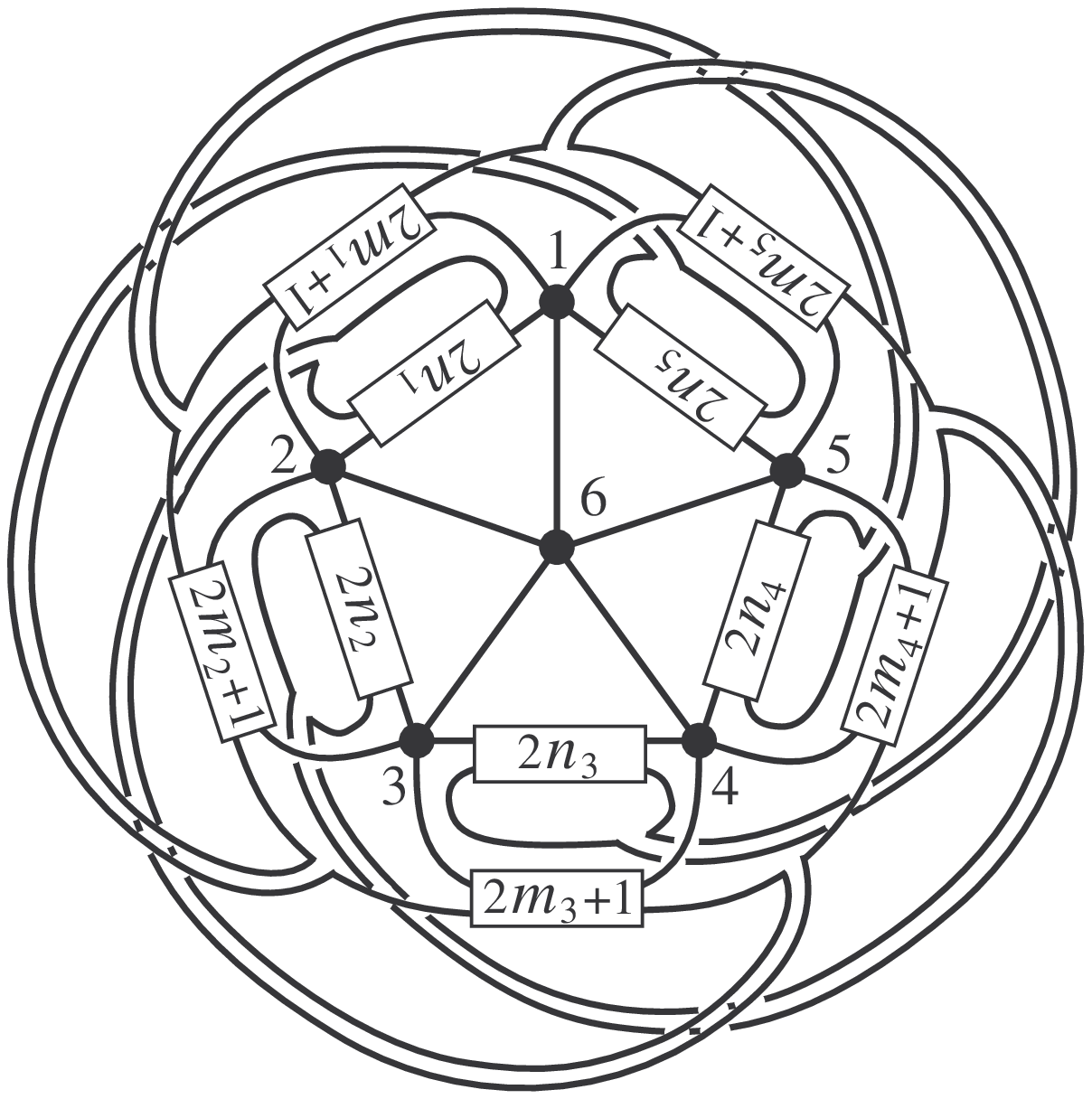}}
      \end{center}
   \caption{}
  \label{K6homology}
\end{figure} 
\begin{figure}[htbp]
      \begin{center}
\scalebox{0.4}{\includegraphics*{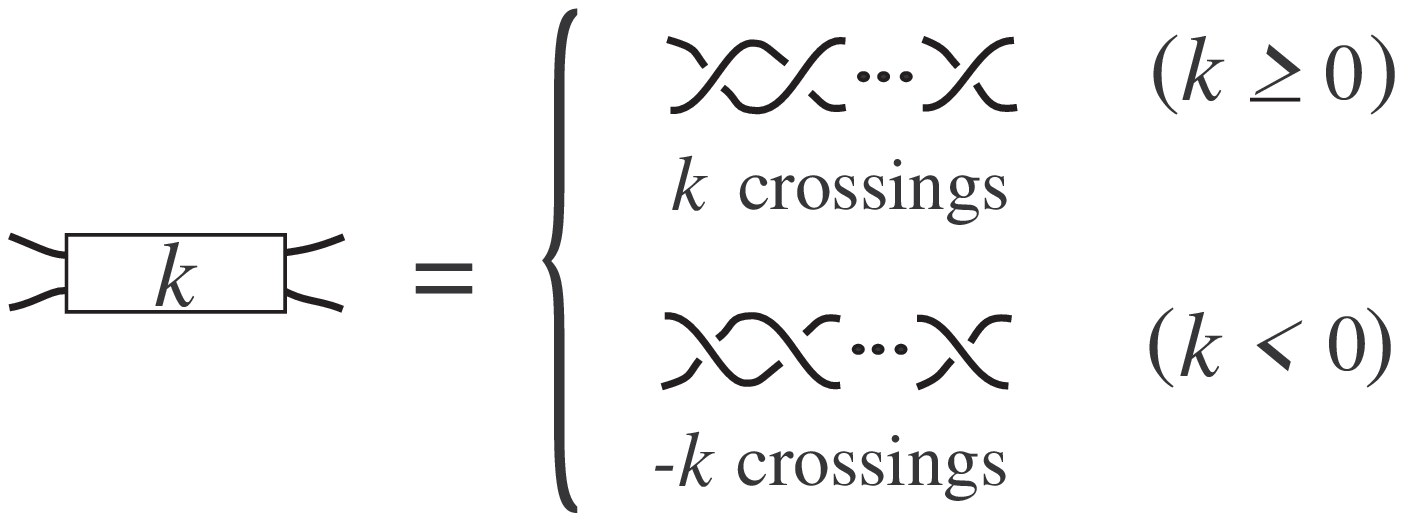}}
      \end{center}
   \caption{}
  \label{twists}
\end{figure} 

We may assume that $h(H_{i})\ (i=1,2,\ldots,5)$ and $h(H_{i+5})\ (i=1,2,\ldots,5)$ are spatial graphs as illustrated in Fig. \ref{K33Simon1} (1) and (2) respectively, $h(G_{i})\ (i=1,2,\ldots,5)$ and $h(G_{6})$ are spatial graphs as illustrated in Fig. \ref{K5Simon1} (1) and (2) respectively, and $h(\lambda_{i})\ (i=1,2,\ldots,5)$ and $h(\lambda_{i+5})\ (i=1,2,\ldots,5)$ are $2$-component links as illustrated in Fig. \ref{lk1} (1) and (2) respectively, where we regard $m_{i+5}=m_{i}$ and $n_{i+5}=n_{i}$ for $i=1,2,\ldots,5$. 

\begin{figure}[htbp]
      \begin{center}
\scalebox{0.4}{\includegraphics*{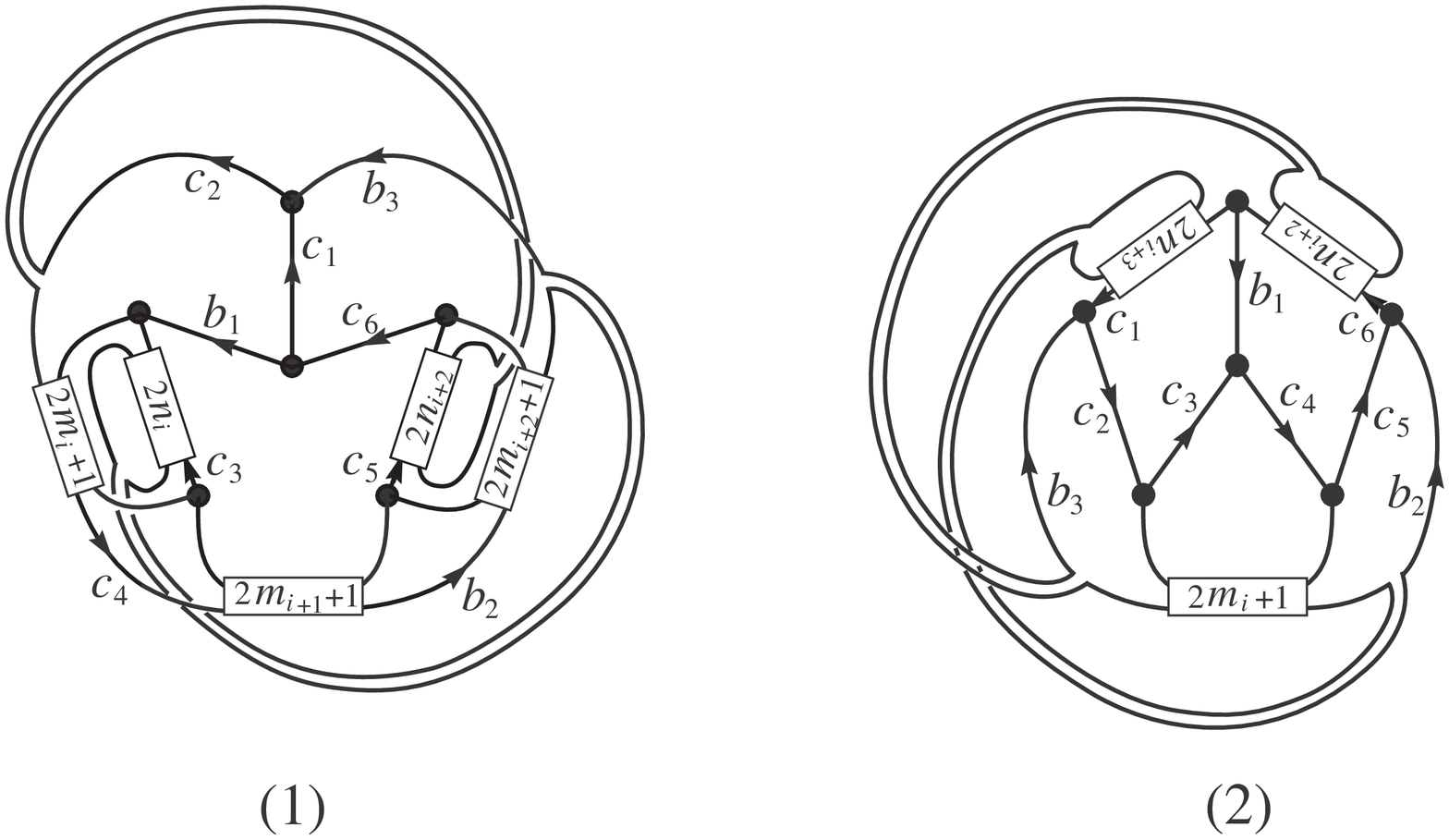}}
      \end{center}
   \caption{}
  \label{K33Simon1}
\end{figure} 
\begin{figure}[htbp]
      \begin{center}
\scalebox{0.4}{\includegraphics*{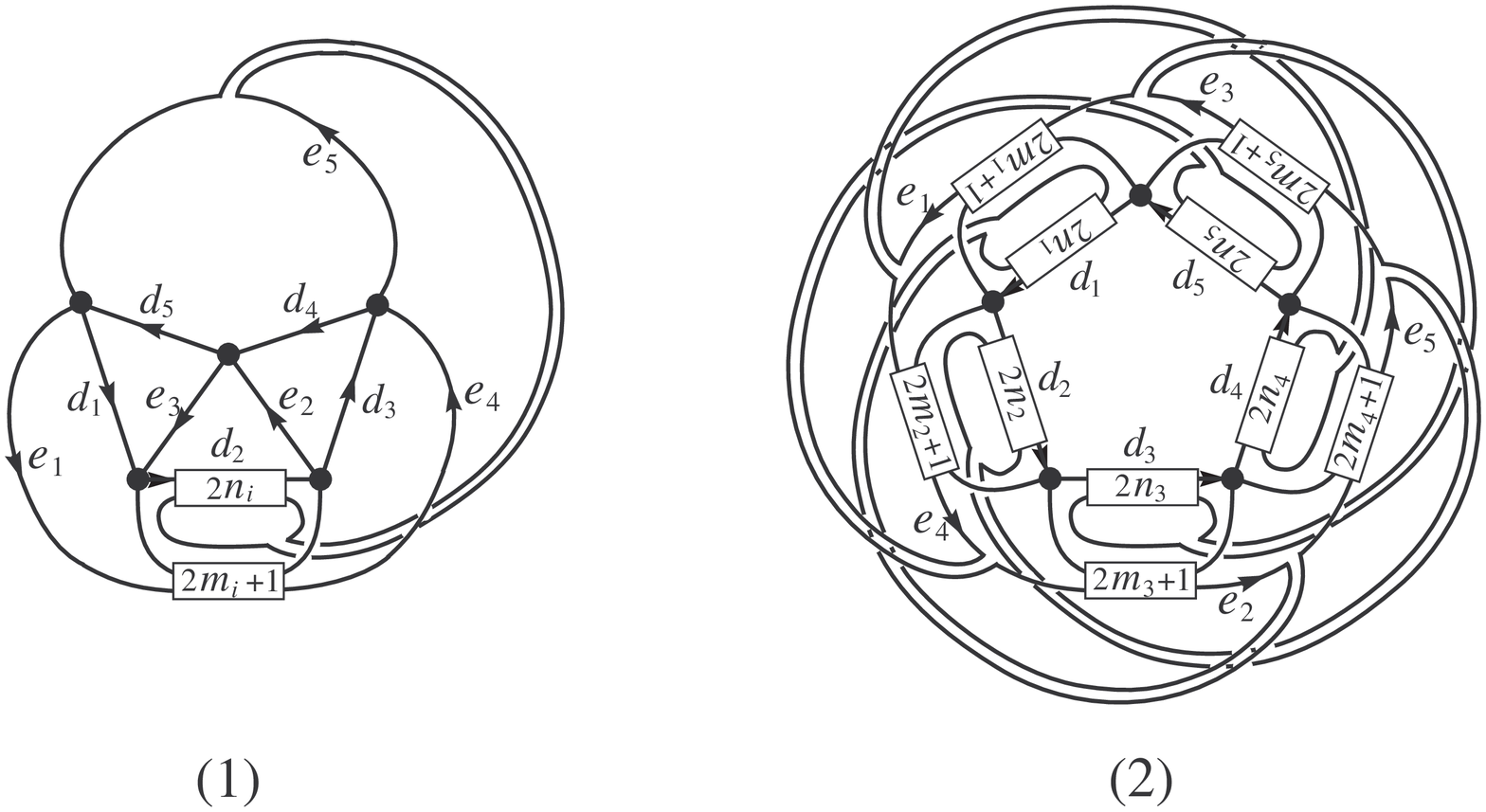}}
      \end{center}
   \caption{}
  \label{K5Simon1}
\end{figure} 
\begin{figure}[htbp]
      \begin{center}
\scalebox{0.4}{\includegraphics*{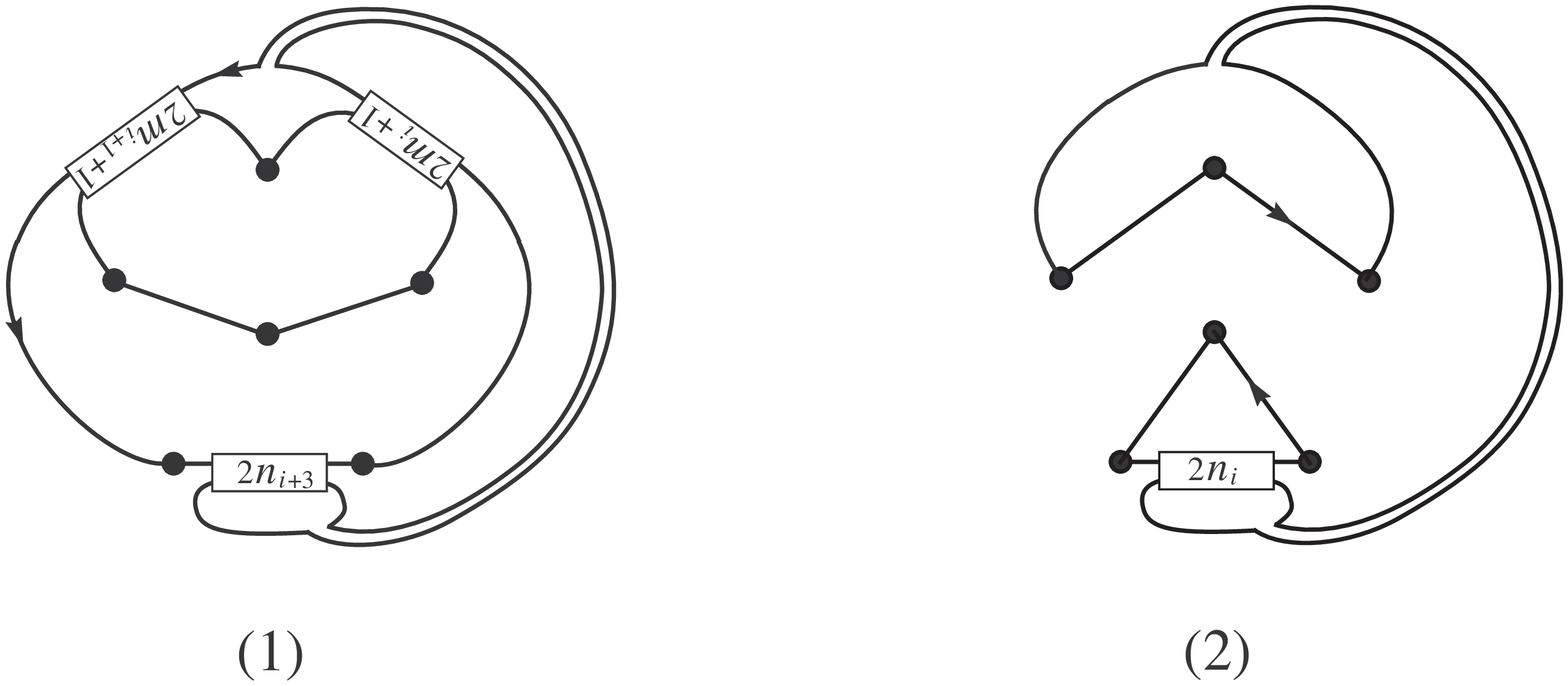}}
      \end{center}
   \caption{}
  \label{lk1}
\end{figure} 

Note that the square of the linking number is an invariant of a non-oriented $2$-component link, and the square of the Simon invariant is also an invariant of non-labeled spatial embeddings of $K_{5}$ and $K_{3,3}$ (see \cite[Lemma 3.2]{NT09}). Thus we may calculate ${\mathcal L}(f|_{H_{i}})^{2}$ and ${\mathcal L}(f|_{G_{i}})^{2}$ by assigning the orientations of edges as illustrated in Fig. \ref{K33Simon1} and \ref{K5Simon1}, respectively, and ${\rm lk}(f(\lambda_{i}))^{2}$ by assigning the orientations of components as illustrated in Fig. \ref{lk1}. Then we have that
\begin{eqnarray*}
{\mathcal L}(h|_{H_{i}})^{2}
&=&
\left\{
-2\left(m_{i}+m_{i+1}+m_{i+2}-n_{i}-n_{i+2}\right)-3
\right\}^{2}\ (i=1,2,\ldots,5),\\
{\mathcal L}(h|_{H_{i+5}})^{2}
&=&
\left\{
2\left(m_{i}-n_{i+3}-n_{i+2}\right)+1
\right\}^{2}\ (i=1,2,\ldots,5), \\
{\mathcal L}(h|_{G_{i}})^{2}
&=&
\left(-2n_{i}+2m_{i}+1\right)^{2}\ (i=1,2,\ldots,5),\\
{\mathcal L}(h|_{G_{6}})^{2}
&=&
\left(
2\sum_{i=1}^{5}m_{i}-2\sum_{i=1}^{5}n_{i}+5
\right)^{2}, \\
{\rm lk}(h(\lambda_{i}))^{2}
&=&
\left(
-m_{i}-m_{i+1}+n_{i+3}-1
\right)^{2}\ (i=1,2,\ldots,5), \\
{\rm lk}(h(\lambda_{i+5}))^{2}
&=&
n_{i}^{2}\ (i=1,2,\ldots,5). 
\end{eqnarray*}
Then by a direct calculation, it is not hard to see that (\ref{f-h}) holds. 
\end{proof}

\begin{proof}[Proof of Theorem \ref{main1}.] 
By Proposition \ref{homo2} (1), we have that 
\begin{eqnarray}
{\mathcal L}(f|_{G_{i}})^{2} &=& 8\alpha_{\omega}(f|_{G_{i}})+1\ (i=1,2,\ldots,6),\label{simon_formula1}\\
{\mathcal L}(f|_{H_{i}})^{2} &=& 8\alpha_{\omega}(f|_{H_{i}})+1\ (i=1,2,\ldots,10). \label{simon_formula2}
\end{eqnarray}
Thus by (\ref{simon_formula1}), (\ref{simon_formula2}) and Lemma \ref{Simon_lemma}, we have that 
\begin{eqnarray}
\sum_{i=1}^{10}{\rm lk}(f(\lambda_{i}))^{2}
&=&\frac{1}{4}\sum_{i=1}^{10}\left\{
8\alpha_{\omega}(f|_{H_{i}})+1
\right\}
-\frac{1}{4}\sum_{i=1}^{6}\left\{
8\alpha_{\omega}(f|_{G_{i}})+1
\right\}\label{alpha_sum}\\
&=&
2\sum_{i=1}^{10}\alpha_{\omega}(f|_{H_{i}})
-2\sum_{i=1}^{6}\alpha_{\omega}(f|_{G_{i}})
+1\nonumber\\
&=&2\sum_{i=1}^{10}
\left(
\sum_{\gamma\in\Gamma_{6}(H_{i})}a_{2}(f(\gamma))
-\sum_{\gamma\in\Gamma_{4}(H_{i})}a_{2}(f(\gamma))
\right)\nonumber\\
&&
-2\sum_{i=1}^{6}
\left(
\sum_{\gamma\in\Gamma_{5}(G_{i})}a_{2}(f(\gamma))
-\sum_{\gamma\in\Gamma_{4}(G_{i})}a_{2}(f(\gamma))
\right)+1. \nonumber
\end{eqnarray}
Then we can see that for any $6$-cycle (resp. $5$-cycle) $\gamma$ of $K_{6}$ there exists exactly one $H_{i}$ (resp. $G_{i}$) such that $\gamma$ is a $6$-cycle of $H_{i}$ (resp. $5$-cycle of $G_{i}$), and 
for any $4$-cycle $\gamma$ of $K_{6}$ there exist exactly two $H_{i}$'s (resp. $G_{i}$'s) such that $\gamma$ is a $4$-cycle of $H_{i}$ (resp. $G_{i}$). Thus we have that 
\begin{eqnarray}
\sum_{i=1}^{10}\sum_{\gamma\in\Gamma_{6}(H_{i})}a_{2}(f(\gamma))
&=&\sum_{\gamma\in \Gamma_{6}(K_{6})}a_{2}(f(\gamma)),\label{alpha_sum2}\\
\sum_{i=1}^{6}\sum_{\gamma\in\Gamma_{5}(G_{i})}a_{2}(f(\gamma))
&=&\sum_{\gamma\in \Gamma_{5}(K_{6})}a_{2}(f(\gamma)),\label{alpha_sum3}\\
\sum_{i=1}^{10}\sum_{\gamma\in\Gamma_{4}(H_{i})}a_{2}(f(\gamma))&=&\sum_{i=1}^{6}\sum_{\gamma\in\Gamma_{4}(G_{i})}a_{2}(f(\gamma))
=2\sum_{\gamma\in \Gamma_{4}(K_{6})}a_{2}(f(\gamma)). \label{alpha_sum4}
\end{eqnarray}
Therefore by (\ref{alpha_sum}), (\ref{alpha_sum2}), (\ref{alpha_sum3}) and (\ref{alpha_sum4}), we have the desired conclusion. 
\end{proof}

Next we give a proof of Theorem \ref{main2}. In the following we denote a path of length $2$ of $K_{7}$ consisting of two edges $\overline{ij}$ and $\overline{jk}$ by $\overline{ijk}$, and the subgraph of $K_{7}$ obtained from $K_{7}$ by deleting the vertex $i$ and all of the edges incident to $i$ by $K_{6}^{(i)}\ (i=1,2,\ldots,7)$. Actually $K_{6}^{(i)}$ is isomorphic to $K_{6}$ for any $i$. 

\begin{proof}[Proof of Theorem \ref{main2}.] 
For $2\le i<j\le 7$, let $F_{ij}$ be the subgraph of $K_{7}$ obtained from $K_{7}$ by deleting the edges $\overline{ij}$ and $\overline{1k}\ (2\le k\le 7,\ k\neq i,j)$. Note that $F_{ij}$ is homeomorphic to $K_{6}$, namely $F_{ij}$ is obtained from the graph isomorphic to $K_{6}$ by subdividing an edge by the vertex $1$, see Fig. \ref{K6subdivide}. 

\begin{figure}[htbp]
      \begin{center}
\scalebox{0.425}{\includegraphics*{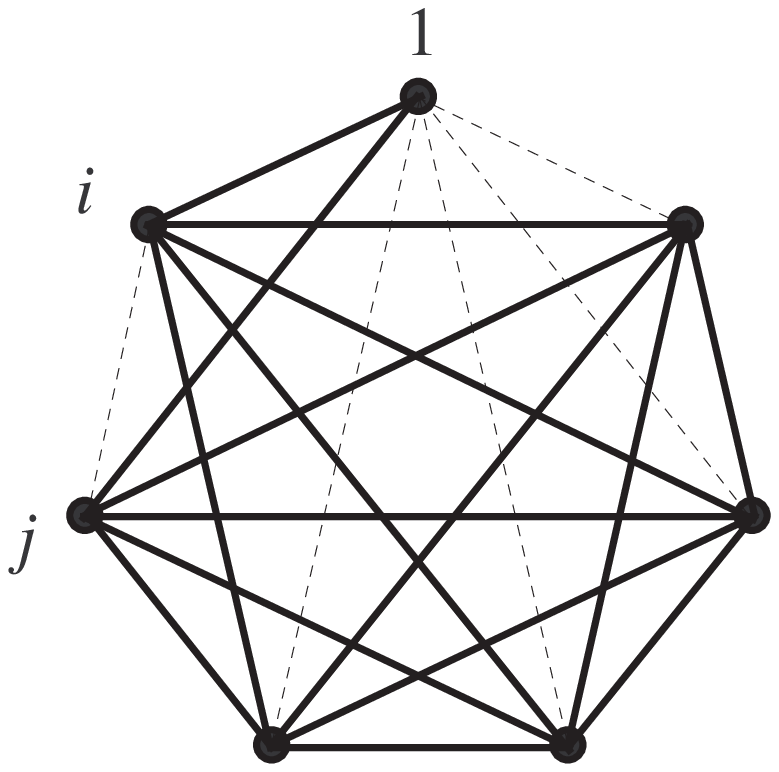}}
      \end{center}
   \caption{}
  \label{K6subdivide}
\end{figure} 

Let $f$ be a spatial embedding of $K_{7}$. Then by applying Theorem \ref{main1} to $f|_{F_{ij}}$, we have that 
\begin{eqnarray}
&&\sum_{\gamma\in\Gamma_{7}(F_{ij})}a_{2}(f(\gamma))
+\sum_{\substack{\gamma\in\Gamma_{6}(F_{ij}) \\ \overline{i1j}\not\subset \gamma}} a_{2}(f(\gamma))
-\sum_{\substack{\gamma\in\Gamma_{6}(F_{ij}) \\ \overline{i1j}\subset \gamma}} a_{2}(f(\gamma))
-\sum_{\substack{\gamma\in\Gamma_{5}(F_{ij}) \\ \overline{i1j}\not\subset \gamma}} a_{2}(f(\gamma))\label{eq1}\\
&=&
\frac{1}{2}\left\{
\sum_{\substack{\lambda=\gamma\cup \gamma'\in\Gamma_{4,3}(F_{ij}) \\ \gamma\in \Gamma_{4}(F_{ij}),\ \gamma'\in \Gamma_{3}(F_{ij}) \\ \overline{i1j}\subset \gamma}} {\rm lk}(f(\lambda))^{2}
+\sum_{\lambda\in\Gamma_{3,3}(F_{ij})} {\rm lk}(f(\lambda))^{2}
-1
\right\}. \nonumber
\end{eqnarray}
Let us take the sum of both sides of (\ref{eq1}) over $2\le i<j\le 7$. For a $7$-cycle $\gamma$ of $K_{7}$, let $i$ and $j$ be the two vertices of $K_{7}$ which are adjacent to $1$ in $\gamma$ ($2\le i<j\le 7$). Then $\gamma$ is a $7$-cycle of $F_{ij}$. This implies that 
\begin{eqnarray}
\sum_{2\le i<j\le 7}\sum_{\gamma\in\Gamma_{7}(F_{ij})}a_{2}(f(\gamma))
&=& \sum_{\gamma\in\Gamma_{7}(K_{7})}a_{2}(f(\gamma)).\label{eq2}
\end{eqnarray}
For a $6$-cycle $\gamma$ of $K_{6}^{(1)}$, let $\overline{ij}$ be an edge of $K_{6}^{(1)}$ which is not contained in $\gamma$. Then $\gamma$ is a $6$-cycle of $F_{ij}$ which does not contain $\overline{i1j}$. Note that there are nine ways to choose such a pair of $i$ and $j$. This implies that 
\begin{eqnarray}
\sum_{2\le i<j\le 7}\sum_{\substack{\gamma\in\Gamma_{6}(F_{ij}) \\ \overline{i1j}\not\subset \gamma}} a_{2}(f(\gamma))
&=& 9\sum_{\gamma\in\Gamma_{6}(K_{6}^{(1)})}a_{2}(f(\gamma)).\label{eq3}
\end{eqnarray}
For a $6$-cycle $\gamma$ of $K_{7}$ which contains the vertex $1$, let $i$ and $j$ be the two vertices of $K_{7}$ which are adjacent to $1$ in $\gamma$. Then $\gamma$ is a $6$-cycle of $F_{ij}$ which contains $\overline{i1j}$. This implies that 
\begin{eqnarray}
\sum_{2\le i<j\le 7}\sum_{\substack{\gamma\in\Gamma_{6}(F_{ij}) \\ \overline{i1j}\subset \gamma}} a_{2}(f(\gamma))
&=& \sum_{\substack{\gamma\in\Gamma_{6}(K_{7}) \\ 1\subset \gamma}} a_{2}(f(\gamma)). \label{eq4}
\end{eqnarray}
For a $5$-cycle $\gamma$ of $K_{6}^{(1)}$, let $\overline{ij}$ be an edge of $K_{6}^{(1)}$ which is not contained in $\gamma$. Then $\gamma$ is a $5$-cycle of $F_{ij}$ which does not contain $\overline{i1j}$. Note that there are ten ways to choose such a pair of $i$ and $j$. This implies that 
\begin{eqnarray}
\sum_{2\le i<j\le 7}\sum_{\substack{\gamma\in\Gamma_{5}(F_{ij}) \\ \overline{i1j}\not\subset \gamma}} a_{2}(f(\gamma))
&=& 10\sum_{\gamma\in\Gamma_{5}(K_{6}^{(1)})}a_{2}(f(\gamma)).\label{eq5}
\end{eqnarray}
For a pair of disjoint cycles $\lambda$ of $K_{7}$ consisting of a $4$-cycle $\gamma$ which contains the vertex $1$ and a $3$-cycle $\gamma'$, let $i$ and $j$ be the two vertices of $K_{7}$ which are adjacent to $1$ in $\gamma$. Then $\lambda$ is a pair of disjoint cycles of $F_{ij}$ consisting of a $4$-cycle $\gamma$ which contains $\overline{i1j}$ and a $3$-cycle $\gamma'$. This implies that 
\begin{eqnarray}
\sum_{2\le i<j\le 7}\sum_{\substack{\lambda=\gamma\cup \gamma'\in\Gamma_{4,3}(F_{ij}) \\ \gamma\in \Gamma_{4}(F_{ij}),\ \gamma'\in \Gamma_{3}(F_{ij}) \\ \overline{i1j}\subset \gamma}} {\rm lk}(f(\lambda))^{2}
&=& \sum_{\substack{\lambda=\gamma\cup \gamma'\in\Gamma_{4,3}(K_{7}) \\ \gamma\in \Gamma_{4}(K_{7}),\ \gamma'\in \Gamma_{3}(K_{7}) \\ 1\subset \gamma}} {\rm lk}(f(\lambda))^{2}.\label{eq6}
\end{eqnarray}
For a pair of disjoint $3$-cycles $\lambda$ of $K_{6}^{(1)}$, let $\overline{ij}$ be an edge of $K_{6}^{(1)}$ which is not contained in $\lambda$. Then $\lambda$ is a pair of disjoint $3$-cycles of $F_{ij}$ which does not contain $\overline{i1j}$. Note that there are nine ways to choose such a pair of $i$ and $j$. This implies that 
\begin{eqnarray}
\sum_{2\le i<j\le 7}\sum_{\substack{\lambda\in\Gamma_{3,3}(F_{ij}) \\ \overline{i1j}\not\subset \lambda}} {\rm lk}(f(\lambda))^{2}
&=& 9\sum_{\lambda\in\Gamma_{3,3}(K_{6}^{(1)})}{\rm lk}(f(\lambda))^{2}.\label{eq7}
\end{eqnarray}
Thus by (\ref{eq1}), (\ref{eq2}), (\ref{eq3}), (\ref{eq4}), (\ref{eq5}), (\ref{eq6}) and (\ref{eq7}), we have that
\begin{eqnarray}
&&\sum_{\gamma\in\Gamma_{7}(K_{7})}a_{2}(f(\gamma))
+9\sum_{\gamma\in\Gamma_{6}(K_{6}^{(1)})}a_{2}(f(\gamma))\label{eq8}\\
&&-\sum_{\substack{\gamma\in\Gamma_{6}(K_{7}) \\ 1\subset \gamma}} a_{2}(f(\gamma))
-10\sum_{\gamma\in\Gamma_{5}(K_{6}^{(1)})}a_{2}(f(\gamma))\nonumber\\
&=& 
\frac{1}{2}\left\{
\sum_{\substack{\lambda=\gamma\cup \gamma'\in\Gamma_{4,3}(K_{7}) \\ \gamma\in \Gamma_{4}(K_{7}),\ \gamma'\in \Gamma_{3}(K_{7}) \\ 1\subset \gamma}} {\rm lk}(f(\lambda))^{2}+9\sum_{\lambda\in\Gamma_{3,3}(K_{6}^{(1)})}{\rm lk}(f(\lambda))^{2}
-15
\right\}\nonumber.
\end{eqnarray}
Then, by applying Theorem \ref{main1} to $f|_{K_{6}^{(1)}}$, we have that 
\begin{eqnarray}
&&9\sum_{\gamma\in\Gamma_{6}(K_{6}^{(1)})}a_{2}(f(\gamma))
-10\sum_{\gamma\in\Gamma_{5}(K_{6}^{(1)})}a_{2}(f(\gamma))\label{eq9}\\
&=& 9\left\{
\sum_{\gamma\in\Gamma_{6}(K_{6}^{(1)})}a_{2}(f(\gamma))
-\sum_{\gamma\in\Gamma_{5}(K_{6}^{(1)})}a_{2}(f(\gamma))
\right\}
-\sum_{\gamma\in\Gamma_{5}(K_{6}^{(1)})}a_{2}(f(\gamma))\nonumber\\
&=& \frac{9}{2}\left\{
\sum_{\lambda\in\Gamma_{3,3}(K_{6}^{(1)})}{\rm lk}(f(\lambda))^{2}
-1
\right\}
-\sum_{\gamma\in\Gamma_{5}(K_{6}^{(1)})}a_{2}(f(\gamma)).\nonumber
\end{eqnarray}
By combining (\ref{eq8}) and (\ref{eq9}), we have that 
\begin{eqnarray*}
&&\sum_{\gamma\in\Gamma_{7}(K_{7})}a_{2}(f(\gamma))
-\sum_{\substack{\gamma\in\Gamma_{6}(K_{7}) \\ 1\subset \gamma}} a_{2}(f(\gamma))
-\sum_{\gamma\in\Gamma_{5}(K_{6}^{(1)})}a_{2}(f(\gamma))\\
&=& 
\frac{1}{2}\left\{
\sum_{\substack{\lambda=\gamma\cup \gamma'\in\Gamma_{4,3}(K_{7}) \\ \gamma\in \Gamma_{4}(K_{7}),\ \gamma'\in \Gamma_{3}(K_{7}) \\ 1\subset \gamma}} {\rm lk}(f(\lambda))^{2}-6
\right\}.
\end{eqnarray*}
Note that this also implies that 
\begin{eqnarray}
&&\sum_{\gamma\in\Gamma_{7}(K_{7})}a_{2}(f(\gamma))
-\sum_{\substack{\gamma\in\Gamma_{6}(K_{7}) \\ i\subset \gamma}} a_{2}(f(\gamma))
-\sum_{\gamma\in\Gamma_{5}(K_{6}^{(i)})}a_{2}(f(\gamma))\label{eq10}\\
&=& 
\frac{1}{2}\left\{
\sum_{\substack{\lambda=\gamma\cup \gamma'\in\Gamma_{4,3}(K_{7}) \\ \gamma\in \Gamma_{4}(K_{7}),\ \gamma'\in \Gamma_{3}(K_{7}) \\ i\subset \gamma}} {\rm lk}(f(\lambda))^{2}-6
\right\}\ (i=1,2,\ldots,7). \nonumber
\end{eqnarray}
Now we take the sum of both sides of (\ref{eq10}) over $i=1,2,\ldots,7$. For a $6$-cycle $\gamma$ of $K_{7}$, let $i$ be a vertex of $K_{7}$ which is contained in $\gamma$. Note that there are six ways to choose such a vertex $i$. This implies that 
\begin{eqnarray}
\sum_{i=1}^{7}\sum_{\substack{\gamma\in\Gamma_{6}(K_{7}) \\ i\subset \gamma}} a_{2}(f(\gamma))
&=& 6\sum_{\gamma\in \Gamma_{6}(K_{7})}a_{2}(f(\gamma)).\label{eq11}
\end{eqnarray}
For a $5$-cycle $\gamma$ of $K_{7}$, let $i$ be a vertex of $K_{7}$ which is not contained in $\gamma$. Then $\gamma$ is a $5$-cycle of $K_{6}^{(i)}$. Note that there are two ways to choose such a vertex $i$. This implies that 
\begin{eqnarray}
\sum_{i=1}^{7}\sum_{\gamma\in\Gamma_{5}(K_{6}^{(i)})}a_{2}(f(\gamma))
&=& 2\sum_{\gamma\in \Gamma_{5}(K_{7})}a_{2}(f(\gamma)).\label{eq12}
\end{eqnarray}
For a pair of disjoint cycles $\lambda$ of $K_{7}$ which consisting of a $4$-cycle $\gamma$ and a $3$-cycle $\gamma'$, let $i$ be a vertex of $K_{7}$ which is contained in $\gamma$. Note that there are four ways to choose such a vertex $i$. This implies that 
\begin{eqnarray}
\sum_{i=1}^{7}\sum_{\substack{\lambda=\gamma\cup \gamma'\in\Gamma_{4,3}(K_{7}) \\ \gamma\in \Gamma_{4}(K_{7}),\ \gamma'\in \Gamma_{3}(K_{7}) \\ i\subset \gamma}} {\rm lk}(f(\lambda))^{2}
&=& 4\sum_{\lambda\in \Gamma_{4,3}(K_{7})}{\rm lk}(f(\lambda))^{2}. \label{eq13}
\end{eqnarray}
Finally, by combining (\ref{eq10}), (\ref{eq11}), (\ref{eq12}) and (\ref{eq13}), we have the desired conclusion. 
\end{proof}

To prove Corollary \ref{main3}, we show the following lemma. 

\begin{Lemma}\label{main3lemma} 
For any spatial embedding $f$ of $K_{7}$, we have that  
\begin{eqnarray*}
2\left\{
\sum_{\gamma\in \Gamma_{6}(K_{7})}a_{2}(f(\gamma))
-2\sum_{\gamma\in \Gamma_{5}(K_{7})}a_{2}(f(\gamma))
\right\}
=
\sum_{\lambda\in \Gamma_{3,3}(K_{7})}{{\rm lk}(f(\lambda))}^{2}-7.
\end{eqnarray*}
\end{Lemma}

\begin{proof}
By applying Theorem \ref{main1} to $f|_{K_{6}^{(i)}}$, we have that 
\begin{eqnarray}
&&2\left\{
\sum_{\gamma\in \Gamma_{6}(K_{6}^{(i)})}a_{2}(f(\gamma))
-\sum_{\gamma\in \Gamma_{5}(K_{6}^{(i)})}a_{2}(f(\gamma))
\right\}\label{eq14}\\
&=&
\sum_{\lambda\in \Gamma_{3,3}(K_{6}^{(i)})}{{\rm lk}(f(\lambda))}^{2}-1.\nonumber
\end{eqnarray}
Now we take the sum of both sides of (\ref{eq14}) over $i=1,2,\ldots,7$. It is clear that 
\begin{eqnarray}
\sum_{i=1}^{7}\sum_{\gamma\in \Gamma_{6}(K_{6}^{(i)})}a_{2}(f(\gamma))
&=&\sum_{\gamma\in \Gamma_{6}(K_{7})}a_{2}(f(\gamma)),\label{eq15}\\
\sum_{i=1}^{7}\sum_{\gamma\in \Gamma_{5}(K_{6}^{(i)})}a_{2}(f(\gamma))
&=&2\sum_{\gamma\in \Gamma_{5}(K_{7})}a_{2}(f(\gamma)),\label{eq16}\\
\sum_{i=1}^{7}\sum_{\lambda\in \Gamma_{3,3}(K_{6}^{(i)})}{{\rm lk}(f(\lambda))}^{2}
&=&\sum_{\lambda\in \Gamma_{3,3}(K_{7})}{{\rm lk}(f(\lambda))}^{2}.\label{eq17}
\end{eqnarray}
Thus by combining (\ref{eq14}), (\ref{eq15}), (\ref{eq16}) and (\ref{eq17}), we have the result. 
\end{proof}

\begin{proof}[Proof of Corollary \ref{main3}.] 
This follows from Theorem \ref{main2} and Lemma \ref{main3lemma}. 
\end{proof}

\section{Applications to rectilinear spatial graphs} 

In this section we give a proof of Theorem \ref{main4}, and an alternative proof of Theorem \ref{HJ}. To prove Theorem \ref{main4}, we need the following result 

\begin{Lemma}\label{K7link} {\rm (Fleming-Mellor \cite{FM09})} 
For any spatial embedding $f$ of $K_{7}$, there exist seven pairs of disjoint $3$-cycles $\lambda_{1},\lambda_{2},\ldots,\lambda_{7}$ of $K_{7}$ and fourteen pairs of disjoint cycles $\lambda_{8},\lambda_{9},\ldots,\lambda_{21}$ of $K_{7}$ each of which consists of a $4$-cycle and a $3$-cycle such that ${\rm lk}(f(\lambda_{i}))\equiv 1 \pmod{2}$ for $i=1,2,\ldots,21$. 
\end{Lemma}

Lemma \ref{K7link} implies that for any spatial embedding $f$ of $K_{7}$, $f(K_{7})$ contains at least twenty one non-splittable $2$-component links. Then we have the following. 

\begin{Lemma}\label{recti_key} 
Let $f$ be a spatial embedding of $K_{7}$. Then we have that 
\begin{eqnarray*}
\sum_{\gamma\in \Gamma_{7}(K_{7})}a_{2}(f(\gamma))
-2\sum_{\gamma\in \Gamma_{5}(K_{7})}a_{2}(f(\gamma))
\ge 1. 
\end{eqnarray*}
\end{Lemma}

\begin{proof}
By Lemma \ref{K7link}, we have that 
\begin{eqnarray}
\sum_{\lambda\in \Gamma_{4,3}(K_{7})}{{\rm lk}(f(\lambda))}^{2}
&\ge& 14,\label{14}\\
\sum_{\lambda\in \Gamma_{3,3}(K_{7})}{{\rm lk}(f(\lambda))}^{2}
&\ge& 7.\label{7}
\end{eqnarray}
Then by (\ref{14}), (\ref{7}) and (\ref{cor2_2}), we have that 
\begin{eqnarray*}
\sum_{\gamma\in \Gamma_{7}(K_{7})}a_{2}(f(\gamma))
-2\sum_{\gamma\in \Gamma_{5}(K_{7})}a_{2}(f(\gamma))
\ge \frac{1}{7}\left(
2\cdot 14 + 3\cdot 7
\right)
-6
=1. 
\end{eqnarray*}
This completes the proof. 
\end{proof}

\begin{proof}[Proof of Theorem \ref{main4}.] 
Let $f$ be a rectilinear spatial embedding of $K_{7}$. Then by Proposition \ref{stick} (1), if $k\le 5$ then $f(\gamma)$ is a trivial knot for any $k$-cycle $\gamma$ of $K_{7}$. Since the Conway polynomial of a trivial knot is equal to $1$, we have that 
\begin{eqnarray}\label{eq21}
\sum_{\gamma\in \Gamma_{5}(K_{7})}a_{2}(f(\gamma))=0. 
\end{eqnarray}
Therefore by (\ref{eq21}) and Lemma \ref{recti_key}, we have that $\sum_{\gamma\in \Gamma_{7}(K_{7})}a_{2}(f(\gamma))\ge 1$. As we have already seen in Theorem \ref{main2}, $\sum_{\gamma\in \Gamma_{7}(K_{7})}a_{2}(f(\gamma))$ is an odd integer. Thus we have that $\sum_{\gamma\in \Gamma_{7}(K_{7})}a_{2}(f(\gamma))$ is a positive odd integer. Next we show the latter half of the theorem. First we show the `if' part. Assume that the non-trivial $2$-component links in $f(K_{7})$ are exactly twenty one Hopf links. Since the linking number of a Hopf link is equal to $\pm 1$, by Lemma \ref{K7link}, we have that 
\begin{eqnarray}
\sum_{\lambda\in \Gamma_{4,3}(K_{7})}{{\rm lk}(f(\lambda))}^{2}
&=& 14,\label{eq19}\\
\sum_{\lambda\in \Gamma_{3,3}(K_{7})}{{\rm lk}(f(\lambda))}^{2}
&=& 7.\label{eq20}
\end{eqnarray}
Then by (\ref{eq21}), (\ref{eq19}), (\ref{eq20}) and (\ref{cor2_2}), we have that $\sum_{\gamma\in \Gamma_{7}(K_{7})}a_{2}(f(\gamma))=1$. Next we show the `only if' part. Assume that $\sum_{\gamma\in \Gamma_{7}(K_{7})}a_{2}(f(\gamma))=1$. Note that any non-trivial $2$-component link in $f(K_{7})$ is either a Hopf link or a $(2,4)$-torus link by Proposition \ref{stick} (2). Let $n_{4,3}(4_{1}^{2})$, $n_{4,3}(2_{1}^{2})$ and $n_{3,3}(2_{1}^{2})$ be the number of $(2,4)$-torus links in $f(K_{7})$ each of which is the image of $\lambda\in \Gamma_{4,3}(K_{7})$, the number of Hopf links in $f(K_{7})$ each of which is the image of $\lambda\in \Gamma_{4,3}(K_{7})$ and the number of Hopf links in $f(K_{7})$ each of which is the image of $\lambda\in \Gamma_{3,3}(K_{7})$, respectively. Note that the linking number of a $(2,4)$-torus link is equal to $\pm 2$. Thus by Lemma \ref{K7link} we have that 
\begin{eqnarray}
n_{4,3}(2_{1}^{2})&\ge& 14,\label{eq22}\\
n_{3,3}(2_{1}^{2})&\ge& 7.\label{eq23}
\end{eqnarray} 
Then by (\ref{eq22}), (\ref{eq23}) and (\ref{cor2_2}), we have that 
\begin{eqnarray*}
1&=&\frac{2\left\{4n_{4,3}(4_{1}^{2})+n_{4,3}(2_{1}^{2})\right\}+3n_{3,3}(2_{1}^{2})}{7}-6\\
&\ge& \frac{2\left\{4n_{4,3}(4_{1}^{2})+14\right\}+21}{7}-6\\
&=& 1+\frac{8}{7}n_{4,3}(4_{1}^{2}). 
\end{eqnarray*}
This implies that $n_{4,3}(4_{1}^{2})=0$, namely the non-trivial $2$-component links in $f(K_{7})$ are only Hopf links. Thus we have that 
\begin{eqnarray}
n_{4,3}(2_{1}^{2})&=&\sum_{\lambda\in \Gamma_{4,3}(K_{7})}{{\rm lk}(f(\lambda))}^{2},\label{eq25}\\
n_{3,3}(2_{1}^{2})&=&\sum_{\lambda\in \Gamma_{3,3}(K_{7})}{{\rm lk}(f(\lambda))}^{2}.\label{eq26}
\end{eqnarray}
Then by combining (\ref{eq25}), (\ref{eq26}) and (\ref{cor2_2}), we have that 
\begin{eqnarray}\label{eq27}
49=2n_{4,3}(2_{1}^{2})+3n_{3,3}(2_{1}^{2}).
\end{eqnarray}
Clearly, (\ref{eq22}), (\ref{eq23}) and (\ref{eq27}) imply that $n_{4,3}(2_{1}^{2})=14$ and $n_{3,3}(2_{1}^{2})=7$. Therefore the non-trivial $2$-component links in $f(K_{7})$ are exactly twenty one Hopf links.
\end{proof}

\begin{Example}\label{ex} 
{\rm 
In the following, we denote a $k$-cycle $\overline{i_{1}i_{2}}\cup \overline{i_{2}i_{3}}\cup \cdots \cup \overline{i_{k-1}i_{k}}\cup \overline{i_{k}i_{1}}$ of $K_{7}$ by $[i_{1}i_{2}\cdots i_{k}]$. In \cite{CG83}, Conway-Gordon exhibited the spatial embedding $f$ of $K_{7}$ which contains exactly one trefoil knot as the unique non-trivial knot in $f(K_{7})$. Then $f$ also may be realized by a rectilinear spatial embedding of $K_{7}$ as illustrated in Fig. \ref{K7_knot}. Actually, $f([1357246])$ is a trefoil knot. 
\begin{figure}[htbp]
      \begin{center}
\scalebox{0.425}{\includegraphics*{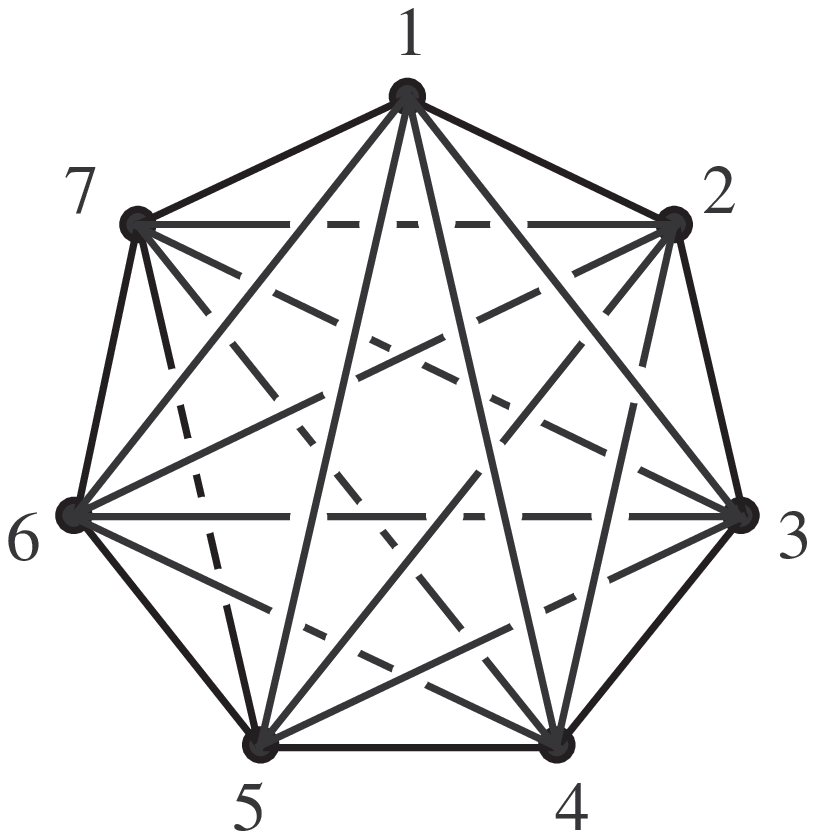}}
      \end{center}
   \caption{}
  \label{K7_knot}
\end{figure} 

Since $\sum_{\gamma\in \Gamma_{7}(K_{7})}a_{2}(f(\gamma))=1$, by Theorem \ref{main4}, we have that the non-trivial $2$-component links in $f(K_{7})$ are exactly twenty one Hopf links. Actually all of the Hopf links in $f(K_{7})$ are as follows: 
\begin{eqnarray*}
&&f([135]\cup [246]),\ f([135]\cup [247]),\ f([136]\cup [247]),\ f([136]\cup [257]),\\
&&f([146]\cup [257]),\ f([146]\cup [357]),\ f([246]\cup [357]),\\
&&f([1246]\cup [357]),\ f([2357]\cup [146]),\ f([1346]\cup [257]),\ f([2457]\cup [136]),\\
&&f([1356]\cup [247]),\ f([2467]\cup [135]),\ f([1357]\cup [246]),\ f([1264]\cup [357]),\\
&&f([2375]\cup [146]),\ f([1436]\cup [257]),\ f([2547]\cup [136]),\ f([1365]\cup [247]),\\
&&f([2476]\cup [135]),\ f([1537]\cup [246]).
\end{eqnarray*}
}
\end{Example}

Next, we give an alternative proof of Theorem \ref{HJ}.

\begin{proof}[Proof of Theorem \ref{HJ}.] 
Let $f$ be a rectilinear spatial embedding of $K_{6}$. As in the proof of (\ref{eq21}), we have that 
\begin{eqnarray}\label{eq28}
\sum_{\gamma\in \Gamma_{5}(K_{7})}a_{2}(f(\gamma))=0. 
\end{eqnarray}
Therefore by (\ref{eq28}) and Theorem \ref{main1}, we have that 
\begin{eqnarray}\label{eq29}
2\sum_{\gamma\in \Gamma_{6}(K_{6})}a_{2}(f(\gamma))
=
\sum_{\lambda\in \Gamma_{3,3}(K_{6})}{{\rm lk}(f(\lambda))}^{2}-1.
\end{eqnarray}

Note that non-trivial knots in $f(K_{6})$ are only trefoil knots, and non-trivial $2$-component links in $f(K_{6})$ are only Hopf links by Proposition \ref{stick}. Let $n_{6}(3_{1})$ and $n_{3,3}(2_{1}^{2})$ be the number of trefoil knots in $f(K_{6})$ each of which is the image of $\lambda\in \Gamma_{6}(K_{6})$ and the number of Hopf links in $f(K_{6})$ each of which is the image of $\lambda\in \Gamma_{3,3}(K_{6})$, respectively. Then we have that
\begin{eqnarray}
n_{6}(3_{1})&=& \sum_{\gamma\in\Gamma_{6}(K_{6})}a_{2}(f(\gamma)),\label{eq30}\\
n_{3,3}(2_{1}^{2})&=& \sum_{\lambda\in \Gamma_{3,3}(K_{6})}{\rm lk}(f(\lambda))^{2}.\label{eq31}
\end{eqnarray} 
As we have already seen in Theorem \ref{main1}, $\sum_{\gamma\in \Gamma_{3,3}(K_{6})}{\rm lk}(f(\lambda))^{2}$ is a positive odd integer. Since $\sharp \Gamma_{3,3}(K_{6})=10$, by (\ref{eq30}), (\ref{eq31}) and (\ref{eq29}), we have that 
\begin{eqnarray}\label{eq32}
(n_{6}(3_{1}),n_{3,3}(2_{1}^{2}))=
(0,1),\ (1,3),\ (2,5),\ (3,7)\ {\rm or}\ (4,9). 
\end{eqnarray}
Assume that $(n_{6}(3_{1}),n_{3,3}(2_{1}^{2}))=(2,5)$. We denote the pairs of disjoint $3$-cycles of $K_{6}$ by $\lambda_{1},\lambda_{2},\ldots,\lambda_{10}$. Then we can see that $\Lambda_{ij}=\lambda_{i}\cup \lambda_{j}$ is isomorphic to the graph illustrated in Fig. \ref{lambda_ij} if $i\neq j$. Note that we can obtain $D_{4}$ by contracting the two edges $e,e'$ of $\Lambda_{ij}$. Then there exists a natural injection
\begin{eqnarray*}
\varphi_{ij}: \Gamma_{4}(D_{4}) \longrightarrow \Gamma(\Lambda_{ij})
\end{eqnarray*}
and a natural bijection 
\begin{eqnarray*}
\psi_{ij}: \Gamma_{2,2}(D_{4}) \longrightarrow \Gamma_{3,3}(\Lambda_{ij}). 
\end{eqnarray*}

\begin{figure}[htbp]
      \begin{center}
\scalebox{0.425}{\includegraphics*{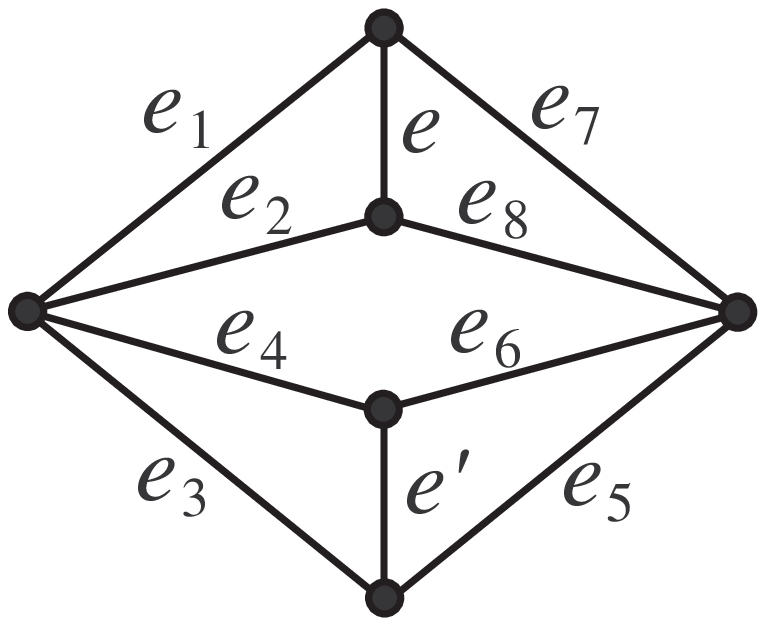}}
      \end{center}
   \caption{}
  \label{lambda_ij}
\end{figure} 

For a spatial embedding $f|_{\Lambda_{ij}}$ of $\Lambda_{ij}$, there exists a spatial embedding $\bar{f}$ of $D_{4}$ such that $\bar{f}(D_{4})$ is obtained from $f(\Lambda_{ij})$ by contracting each of $f(e)$ and $f(e')$ into one point. Note that this embedding is unique up to ambient isotopy in ${\mathbb R}^{3}$. We say that {\it $\bar{f}$ is naturally induced from $f|_{\Lambda_{ij}}$}. Note that for any $4$-cycle $\gamma$ and a pair of disjoint $2$-cycles $\lambda$ of $D_{4}$, $f(\varphi_{ij}(\gamma))$ is ambient isotopic to $\bar{f}(\gamma)$, and $f(\psi_{ij}(\lambda))$ is ambient isotopic to $\bar{f}(\lambda)$. Now we define 
\begin{eqnarray*}
\alpha(f|_{\Lambda_{ij}})=
\sum_{\gamma\in\Gamma(D_{4})}\omega(\gamma)a_{2}(f(\varphi_{ij}(\gamma))), 
\end{eqnarray*}
where $\omega:\Gamma(D_{4})\to {\mathbb Z}$ is the map defined in section $2$. Then, by Proposition \ref{homo2} (2), we have that 
\begin{eqnarray}\label{eq34}
\left|\alpha(f|_{\Lambda_{ij}})\right|
&=&
\left|
\sum_{\gamma\in\Gamma(D_{4})}\omega(\gamma)a_{2}(f(\varphi_{ij}(\gamma)))
\right| \\
&=&
\left|
\sum_{\gamma\in\Gamma(D_{4})}\omega(\gamma)a_{2}(\bar{f}(\gamma))
\right|\nonumber\\
&=& \left|
{\rm lk}(\bar{f}(\lambda)){\rm lk}(\bar{f}(\lambda'))
\right|\nonumber\\
&=& \left|
{\rm lk}(f(\psi_{ij}(\lambda))){\rm lk}(f(\psi_{ij}(\lambda')))
\right|.\nonumber
\end{eqnarray}
If both $f(\lambda_{i})$ and $f(\lambda_{j})$ are Hopf links, then by (\ref{eq34}), we have that 
\begin{eqnarray}\label{eq35}
\left|\alpha(f|_{\Lambda_{ij}})\right|=1. 
\end{eqnarray}
Note that $\varphi_{ij}(D_{4})$ contains all $6$-cycles of $\Lambda_{ij}$, and $\omega(\varphi_{ij}^{-1}(\gamma))=1$ for any $6$-cycle $\gamma$ of $\Lambda_{ij}$. Since the only non-trivial knots in $f(\Lambda_{ij})$ are trefoil knots, by (\ref{eq35}) we have that there exists exactly one $6$-cycle $\gamma_{0}$ of $\Lambda_{ij}$ such that $f(\gamma_{0})$ is a trefoil knot. Recall that there are five Hopf links by assumption. Thus we have that there exist ten $\Lambda_{ij}$'s such that each $f(\Lambda_{ij})$ contains exactly one trefoil knot. It is easy to see that any $6$-cycle of $K_{6}$ is common for exactly three $\Lambda_{ij}$'s. Therefore there are at least four trefoil knots which are contained in $f(K_{6})$. But the number of trefoil knots should be equal to two. This is a contradiction. If we assume that $(n_{6}(3_{1}),n_{3,3}(2_{1}^{2}))=(3,7)$ or $(4,9)$, we can see that a contradiction occurs for each case in the same way as above. Thus we have that 
\begin{eqnarray}\label{eq33}
(n_{6}(3_{1}),n_{3,3}(2_{1}^{2}))\neq 
(2,5),\ (3,7),\ (4,9). 
\end{eqnarray}
Clearly, (\ref{eq32}) and (\ref{eq33}) imply the desired conclusion. 
\end{proof}

\section*{Acknowledgment}

The author is grateful to Professor Masakazu Teragaito for informing him about the result in \cite{RA99}, and Professor Kouki Taniyama for his valuable comments. The author is also grateful to the referee for his or her very careful reading of the manuscript and helpful comments. 

{\normalsize
}


\begin{thebibliography}{99}

\bibitem{ABGW97}
C. C. Adams, B. M. Brennan, D. L. Greilsheimer and A. K. Woo, 
Stick numbers and composition of knots and links, 
{\it J. Knot Theory Ramifications} {\bf 6} (1997), 149--161. 

\bibitem{Natan95}
D. Bar-Natan, 
On the Vassiliev knot invariants, 
{\it Topology} {\bf 34} (1995), 423--472. 

\bibitem{B77}
A. F. Brown, Embeddings of graphs in $E^{3}$, Ph. D. Dissertation, Kent State University, 1977. 

\bibitem{CG83}
J. H. Conway and C. McA. Gordon, 
Knots and links in spatial graphs, 
{\it J. Graph Theory} {\bf 7} (1983), 445--453. 

\bibitem{FM09}
T. Fleming and B. Mellor, 
Counting links in complete graphs, 
{\it Osaka J. Math.} {\bf 46} (2009), 173--201. 

\bibitem{HJ07}
Y. Huh and C. Jeon, 
Knots and links in linear embeddings of $K\sb 6$, 
{\it J. Korean Math. Soc.} {\bf 44} (2007), 661--671. 


\bibitem{Kauffman83}
L. H. Kauffman, 
{\it Formal knot theory}, 
Mathematical Notes, {\bf 30}. {\it Princeton University Press, Princeton, NJ,} 1983. 


\bibitem{Mater02}
A. Mater, 
Stick numbers of links with two trivial components, Summer 2002 Knot Theory Projects, Department of Mathematics, California State University, San Bernardino. 


\bibitem{MT97}
T. Motohashi and K. Taniyama, 
Delta unknotting operation and vertex homotopy of graphs in $R^{3}$, 
{\it KNOTS '96 (Tokyo),} 185--200, {\it World Sci. Publ., River Edge, NJ,} 1997.
\bibitem{Murakami96}
H. Murakami, 
Vassiliev invariants of type two for a link, 
{\it Proc. Amer. Math. Soc.} {\bf 124} (1996), 3889--3896. 


\bibitem{Negami91}
S. Negami, 
Ramsey theorems for knots, links and spatial graphs, 
{\it Trans. Amer. Math. Soc.} {\bf 324} (1991), 527--541. 

\bibitem{N00}
R. Nikkuni, 
The second skew-symmetric cohomology group and spatial embeddings of graphs, 
{\it J. Knot Theory Ramifications} {\bf 9} (2000), 387--411. 

\bibitem{NT09}
R. Nikkuni and K. Taniyama, Symmetries of spatial graphs and Simon invariants, preprint. (arXiv:math.{\tt GT/0708.0066})

\bibitem{OT01}
Y. Ohyama and K. Taniyama, 
Vassiliev invariants of knots in a spatial graph, 
{\it Pacific J. Math.} {\bf 200} (2001), 191--205. 

\bibitem{RA99}
J. L. Ram{\'\i}rez Alfons{\'\i}n, 
Spatial graphs and oriented matroids: the trefoil, 
{\it Discrete Comput. Geom.} {\bf 22} (1999), 149--158. 

\bibitem{Rober65}
R. A. Robertello, 
An invariant of knot cobordism, 
{\it Comm. Pure Appl. Math.} {\bf 18} (1965), 543--555. 

\bibitem{S96}
M. Suzuki, Classification of the spatial-graph homology classes of a complete graph (in Japanese), Master Thesis, Tokyo Denki University, 1996. 

\bibitem{taniyama94a}
K. Taniyama, 
Link homotopy invariants of graphs in $R\sp 3$, 
{\it Rev. Mat. Univ. Complut. Madrid} {\bf 7} (1994), 129--144. 

\bibitem{taniyama94b}
K. Taniyama,  
{Cobordism, homotopy and homology of graphs in $R\sp 3$},  
{\it Topology} {\bf 33} (1994), 509--523. 


\bibitem{TY01}
K. Taniyama and A. Yasuhara, 
Realization of knots and links in a spatial graph,  
{\it Topology Appl.} {\bf 112} (2001), 87--109. 

\bibitem{Vass90}
V. A. Vassiliev, 
Cohomology of knot spaces, 
{\it Theory of singularities and its applications,} 23--69, 
{\it Adv. Soviet Math.,} {\bf 1}, {\it Amer. Math. Soc., Providence, RI,} 1990. 



\end{thebibliography}
\end{document}